\documentclass[12pt,letterpaper,reqno]{amsart}

\usepackage{times}
\usepackage[T1]{fontenc}
\usepackage{mathrsfs}
\usepackage{latexsym}
\usepackage[dvips]{graphics}
\usepackage{epsfig}
\usepackage{amsmath,amsfonts,amsthm,amssymb,amscd}
\input amssym.def
\input amssym.tex

\newcommand\be{\begin{equation}}
\newcommand\ee{\end{equation}}
\newcommand\bea{\begin{eqnarray}}
\newcommand\eea{\end{eqnarray}}
\newcommand\bi{\begin{itemize}}
\newcommand\ei{\end{itemize}}
\newcommand\ben{\begin{enumerate}}
\newcommand\een{\end{enumerate}}
\newcommand\bc{\begin{center}}
\newcommand\ec{\end{center}}
\newcommand\ba{\begin{array}}
\newcommand\ea{\end{array}}

\newtheorem{thm}{Theorem}[section]
\newtheorem{conj}[thm]{Conjecture}

\newtheorem{lem}[thm]{Lemma}

\theoremstyle{definition}
\newtheorem{rek}[thm]{Remark}

\begin{document}

\title{On the possible orders of a basis for a finite cyclic group}

\author{Peter Hegarty}
\email{hegarty@chalmers.se} \address{Mathematical Sciences,
Chalmers University Of Technology and University of Gothenburg,
41296 Gothenburg, Sweden}


\date{\today}

\begin{abstract} We prove a conjecture of Dukes and Herke concerning the 
possible orders of a basis for the cyclic group $\mathbb{Z}_n$, namely : For
each $k \in \mathbb{N}$ there exists a constant $c_k > 0$ such that, for all
$n \in \mathbb{N}$, if $A \subseteq \mathbb{Z}_n$ is a basis of order 
greater than $n/k$, then the order of $A$ is within $c_k$ of $n/l$ for some 
integer $l \in [1,k]$. The proof makes use of 
various results in additive number 
theory concerning the growth of sumsets.  
\end{abstract}


\maketitle

\setcounter{equation}{0}

\setcounter{equation}{0}

\section{Introduction}

Let $G$ be an abelian group, written additively, and $A$ a subset of $G$. 
For a positive integer $h$ we denote by $hA$ the subset of $G$ consisting
of all possible sums of $h$ not necessarily distinct element of $A$, i.e.:
\be
hA = \{a_1 + \cdots + a_h : a_i \in A \}.
\ee
This set is called the {\em $h$-fold sumset} of $A$. We say that $A$ is a 
{\em basis} for $G$ if $hA = G$ for
some $h \in \mathbb{N}$. Define the function
$\rho : 2^{G} \rightarrow \mathbb{N} \cup \{\infty\}$ as follows :
\be
\rho (A) := \left\{ \begin{array}{lr} \min \{h : hA = G\}, & {\hbox{if $A$
is a basis for $G$}}, \\ \infty, & {\hbox{otherwise}}. \end{array} \right.
\ee
In the case where $\rho (A) < \infty$, this invariant is usually referred to 
as the {\em order}{\footnote{In \cite{DH} the term {\em exponent} is
used, while in \cite{KL} the term {\em positive diameter} appears. These 
authors also employ different notations.}} 
of the basis $A$. 
\\
\\
Now let us specialise to the case $G = \mathbb{Z}_n$, a finite cyclic group. 
Throughout this paper we will write $\rho_{n}(A)$ when referring to
a subset $A$ of $\mathbb{Z}_n$. Clearly a subset $A \subseteq \mathbb{Z}_n$ 
is a basis if and only if the greatest common
divisor of its elements is relatively prime to $n$. Also, it is easy to see 
that, if $\rho_{n}(A) < \infty$ then $\rho_{n}(A) \leq n-1$, 
with equality if and only if
$A = \{a_1,a_2\}$ is a 2-element set with GCD$(a_2-a_1,n) = 1$. Hence the 
range of the function $\rho_{n}$ is contained inside $[1,n-1] \cup \{\infty\}$. 
It has been known for some time that, for large enough $n$, the range of 
$\rho_{n}$
does not contain the entire interval of integers $[1,n-1]$. In a recent 
paper, which also contains a summary of previously known results, Dukes 
and Herke \cite{DH} make a conjecture regarding gaps in the range of $\rho_{n}$.
Their hypothesis may be stated precisely as follows :

\begin{conj}
For each $k \in \mathbb{N}$ there exists an absolute constant $c_k > 0$
such that the following holds :
\par For any $n \in \mathbb{N}$, if $A$ is a basis for $\mathbb{Z}_n$ for
which $\rho_{n}(A) \geq n/k$, then there is some integer $l \in [1,k]$ such that
$|\rho_{n} (A) - n/l | < c_k$.
\end{conj}

Observe that the conjecture would imply the existence of arbitrarily long gaps 
in the range of $\rho_{n}$, for all sufficiently large $n$. The purpose of our
note is to prove this conjecture, using results from additive number theory
concerning the structure of sets with small doubling. The idea is roughly
as follows : On the one hand, the conjecture says something about the possible
orders of a basis for $\mathbb{Z}_n$ when that order is large, namely of 
order $n$. On the other hand, various results from additive number theory
imply that if $A$ is a basis for $\mathbb{Z}_n$, then the iterated sumsets
$hA$ cannot grow in size $\lq$too slowly' and, if the growth rate is
close to the slowest possible, then $A$ has a very restricted structure. 
Putting these two things together allows us to describe closely the
structure of (a small multiple of) a basis $A$ of large order, and from there 
we can establish the conjecture.        
   
\setcounter{equation}{0}
\section{Preliminaries}

Here we state three results from the additive number theory literature which 
will be used in our proof of Conjecture 1.1. 

 
The first result is part of Theorem 2.5 of \cite{KL} :


\begin{thm} {\bf (Klopsch-Lev)}
Let $n \in \mathbb{N}$ and $\rho \in [2,n-1]$. Let $A$ be a basis for 
$\mathbb{Z}_n$ such that $\rho_{n}(A) \geq \rho$. Then 
\be
|A| \leq \max \left\{ \frac{n}{d} \left( \lfloor \frac{d-2}{\rho - 1} \rfloor
+ 1 \right) : d | n, \; d \geq \rho + 1 \right\},
\ee
In particular, for each fixed $k \in \mathbb{N}$, if $\rho_{n}(A) \geq n/k$ 
and $n \gg 0$, then $|A| \leq 2k$.
\end{thm}


The second result concerns the structure of subsets of $\mathbb{Z}_n$
with small doubling and is Theorem 1 of \cite{DF} :

\begin{thm} {\bf (Deshouillers-Freiman)}
Let $n \in \mathbb{N}$ and $A$ a non-empty subset of $\mathbb{Z}_n$ such that
$|A| < 10^{-9} n$ and $|2A| < 2.04 |A|$. Then there is a subgroup 
$H \subsetneqq G$ such that one of the following three cases holds :
\par (i) if the number of cosets of $H$ met by $A$, let us call it $s$, is
different from $1$ and $3$, then $A$ is included in an arithmetic progression 
of $l$ cosets modulo $H$ such that
\be (l-1)|H| \leq |2A| - |A|.
\ee
\par (ii) if $A$ meets exactly three cosets of $H$, then (2.2) holds with $l$
replaced by $\min \{l,4\}$.
\par (iii) if $A$ is included in a single coset of $H$, then $|A| > 
10^{-9} |H|$.
\\
\\
Furthermore, when $l \geq 2$, there exists a coset of $H$ which contains
more than $\frac{2}{3} |H|$ elements from $A$, a relation superseded by (2.2)
when $l \geq 4$.
\end{thm}

\begin{rek}
In \cite{DF} the authors remark that they expect that the same
structure theorem holds for larger constants than $2.04$ and $10^{-9}$
respectively. This is known to be the case when $n$ is prime, according to
the so-called {\em Freiman-Vosper theorem}. For a proof of that 
$\lq$classical' result, see Theorem 2.10 in \cite{N}.
\end{rek}

The third and 
last result from the literature that we shall use is a special case of 
a result of Lev \cite{L}, generalising an earlier result of Freiman \cite{F}, 
concerning the growth of sumsets of a large 
subset of an arithmetic progression of integers :

\begin{thm} {\bf (Freiman, Lev)}    
Let $A \subseteq \mathbb{Z}$ satisfy
\be
|A| = n, \;\;\; A \subseteq [0,l], \;\;\; \{0,l\} \subseteq A, \;\;\; 
{\hbox{gcd}}(A) = 1.
\ee
If $2n-3 \geq l$ then, for every $h \in \mathbb{N}$ one has 
\be
|hA| \geq n + (h-1)l.
\ee
\end{thm}

\setcounter{equation}{0}
\section{Proof of Conjecture 1.1}

First some notation. Let $G$ be an abelian group and $A \subseteq G$. For
$g \in G$ we denote 
\be
A + g := \{a+g : a \in A \},
\ee
and for $h \in \mathbb{Z}$ we denote
\be
h \cdot A := \{ ha : a \in A \}.
\ee

\begin{lem}
Let $A \subseteq \mathbb{Z}_n$ and $u,v \in \mathbb{Z}$ such that
$(u,n) = 1$. Then $\rho_{n}(A) = \rho_{n}[(u \cdot A) + v]$.
\end{lem}

\begin{proof}
This is clear.
\end{proof}

\begin{lem}
Conjecture 1.1 holds for bases consisting of at most $3$ elements.
\end{lem}

\begin{proof}
Let $n \in \mathbb{N}$ and $A$ be a basis for $\mathbb{Z}_n$ such that $|A| 
\leq 3$. If $|A| = 1$ then $n=1$, so the Conjecture is vacuous. 
If $|A| = 2$ 
then $\rho_{n}(A) = n-1$, as already noted in the Introduction. Suppose 
$|A| = 3$. From Lemma 3.1 it is easy to deduce that, 
without loss of generality, one of the following cases arises :
\\
\\
(i) $A = \{0,1,t\}$ for some $t \in [2,n-1]$, 
\\
(ii) $A = \{0,a,b\}$ where $a \geq 2$, $a \mid n$ and $(a,b) = 1$.
\\
\\
In case (ii) it is easy to see that
\be
\max \left\{ \frac{n}{a} - 1, a-1 \right\} \leq \rho_{n} (A) \leq 
\left( \frac{n}{a} - 1 \right) + (a-1),
\ee
which in turn is easily seen to imply Conjecture 1.1. It remains to deal with 
case (i). In what follows we adopt the following notation : If
$x \in \mathbb{Z}$ and $n \in \mathbb{N}$ then $||x||_n$ denotes the 
numerically least residue of $x$ modulo $n$, that is, the unique integer
$x_0 \in (n/2,n/2]$ such that $x \equiv x_0 \; ({\hbox{mod $n$}})$. 
\\
\\
So fix $k,t \in \mathbb{N}_{> 1}$ and set $A = \{0,1,t\}$. Let $n \in
\mathbb{N}$, which we think of as being very large. 
We suppose that $\rho_{n}(A) > n/k$ and shall show that
Conjecture 1.1 holds. First of all, by a standard pigeonhole principle
argument, there is some integer $c \in [1,k-1]$ such that 
$||ct||_{n} \leq n/k$. 
Put $r := ||ct||_{n}$ and $s := |r|$. Clearly, the order of 
the basis $\{0,1,s\}$ is at most $s + n/s$ (if $s = 0$ this quantity becomes
infinite, which is consistent with our earlier notations). In terms
of $A$, this implies that
\be
\rho_{n}(A) \leq s + \frac{cn}{s}.
\ee
The function $f(s) = s + cn/s$ has a local minimum at $s = \sqrt{cn}$. Note 
also that $f(ck) = f(n/k) = n/k + ck$. It follows that, for $n \gg 0$, if 
$\rho_{n}(A) > n/k + ck$ then $s \leq ck$. In terms of $t$, the latter
implies that
\be
t = \frac{dn+e}{c},
\ee
for some integers $d \in [0,c)$, $e \in [-ck,ck]$. In this representation
of $t$, we may assume that $(d,c) = 1$. The important 
point is that each of $c,d,e$ is $O(k)$. First suppose $e \geq 0$. 
Clearly then, the number of terms from $A$ needed to represent every number from
$0$ through $n-1$ is at most $O(k)$ greater than the number of terms 
needed to represent every number from $0$ through $\lfloor n/c \rfloor$. 
But since $ct \equiv e \; ({\hbox{mod $n$}})$ it is easy to see in turn that
the latter number of terms is within $O(k)$ of $n/l$, where
$l = \max \{c,e\}$. Thus $| \rho_{n}(A) - n/l| = O(k)$, which implies
Conjecture 1.1.
\par If $e < 0$, then replace $A$ by $1-A = \{0,1,1-t\} \; 
({\hbox{mod $n$}})$ and argue as before. This completes the proof of the 
lemma.    
\end{proof}

We can now complete the proof of Conjecture 1.1. Fix $k \in \mathbb{N}$.
All constants $c_{i,k}$ appearing below depend on $k$ only. Let
$n$ be a positive integer which we think of as being very large. 
Let $A$ be a basis for $\mathbb{Z}_n$ such that $\rho_{n}(A) > n/k$. 
By Lemma 3.1 we may assume, without loss of generality, that $0 \in A$. This
is a convenient assumption, as it implies that $hA \subseteq (h+1)A$ for 
every $h$. 
From Theorem 2.1 it is easy to deduce the existence of positive constants
$c_{1,k}, c_{2,k}$, such that 
\be
|A| \leq c_{1,k}
\ee
and, for some integer $j \in [1,c_{2,k}]$ one must have 
\be
|2^{j+1} A| < 2.04 |2^{j} A|.
\ee
Set $h := 2^j$. For $n$ sufficiently large, we'll certainly have 
$|hA| < 10^{-9} n$ and
so we can apply Theorem 2.2. Let $H$ be the 
corresponding subgroup of $\mathbb{Z}_n$ and $\pi : \mathbb{Z}_n \rightarrow
\mathbb{Z}_n / H$ the natural projection. We can identify $H$ with
$\mathbb{Z}_m$ for some proper divisor $m$ of $n$, and then identify
$\mathbb{Z}_n / H$ with $\mathbb{Z}_{n/m}$. Let $B := h A$. Since
$A$ is a basis for $\mathbb{Z}_n$, then so is $B$ and hence $\pi(B)$ is a 
basis for $\mathbb{Z}_{n/m}$. This means that either case (i) or case (ii)
of Theorem 2.2 must apply. Moreover, since some coset of $H$ contains at 
least $\frac{2}{3} |H|$ elements from $B$, it follows that 
$m = |H| = O(|B|) = O(k)$. Thus 
\be
m \leq c_{3,k}, 
\ee
say. Since 
\be
\rho_{n/m}(\pi(A)) \leq \rho_{n}(A) \leq \rho_{n/m}(\pi(A)) + m,
\ee
this together with (3.7) and (3.8) imply that
\be
|\rho_{n}(A) - h \rho_{n/m}(\pi(B))| \leq c_{4,k}.
\ee
To prove Conjecture 1.1, it thus suffices to show that 
\be
| \rho_{n/m}(\pi(B)) - n/q | \leq c_{5,k}, \;\;\; {\hbox{for some
multiple $q$ of $h$}}.
\ee
Let $s$ be the number of cosets of $H$ met by $B$ and $s^{\prime}$ the
number met by $A$. 
\\
\\
{\sc Case 1} : $s = 3$. 
\\
\\
Then $s^{\prime} \leq 3$. We don't need (3.11) in this case and can
instead deduce Conjecture 1.1 directly from (3.9) and Lemma
3.2.
\\
\\
{\sc Case 2} : $s \neq 3$. 
\\
\\
Then Case (i) of Theorem 2.2 must apply. Let $l$ be the minimum length of an 
arithmetic progression in $\mathbb{Z}_{n/m}$ containing $\pi(B)$. Note that
$l \leq c_{6,k}$, by (2.1). By Lemma 3.1, there is no loss of
generality in assuming that $\pi(B)$ is contained inside an interval of length 
$l-1$. Since $\pi(A) \subseteq \pi(B)$ and $l = O(k)$ we can now also 
see that $l-1$ is a multiple of $h$, provided $n \gg 0$. Thus it suffices
to prove that
\be
\left| \rho_{n/m}(\pi(B)) - \frac{n}{l-1} \right| \leq c_{7,k}.
\ee
It is here that we use Theorem 2.4. Indeed (3.12) is easily seen to follow
from that theorem provided that $2s-3 \geq l-1$. But this inequality
is in turn easily checked to result from (2.1) (as applied to $B$), (3.7)
and the fact that $|B| \leq s|H|$. 
\par Thus the proof of Conjecture 1.1 is complete.     

\setcounter{equation}{0}
\section{Concluding remarks}

Explicit values for each of the constants $c_{i,k}$,  $i = 1,...,7$, can 
easily be extracted from the argument given above. 
Similarly, one can extract explicit bounds for all the $O(k)$ terms in
the proof of Lemma 3.2. All of this will in
turn yield explicit constants $c_k$ in Conjecture 1.1. We refrain from 
carrying out this messy procedure, however, since the more interesting
question is what the optimal values are for the $c_k$. Note that
$c_k \geq (k-2) + \frac{1}{k}$, which can be seen by considering the basis
$\{0,1,k\}$ for $\mathbb{Z}_n$, when $n \equiv -1 \; ({\hbox{mod $k$}})$. 

\section*{Acknowledgement}
I thank Renling Jin for very helpful discussions.

\ \\

\end{document}